\documentclass[12pt]{amsart}

\usepackage[utf8]{inputenc}
\usepackage{mathrsfs}
\usepackage{amssymb,amsmath,amscd}
\usepackage{setspace}
\usepackage{color}
\usepackage[all]{xy}
\usepackage{graphics}
\usepackage[english]{babel}
\usepackage{url}

\usepackage{hyperref}

\newtheorem{Theorem}{Theorem}[section]
\newtheorem{Lemma}[Theorem]{Lemma}
\newtheorem{Proposition}[Theorem]{Proposition}
\newtheorem{corollary}[Theorem]{Corollary}

\theoremstyle{remark}
\newtheorem{remark}{Remark}

\title[]{Some results on compact K\"ahler manifolds with elliptic homotopy type}

\author{Yang Su}
\address{School of Mathematics and Systems Science, Chinese Academy of Sciences, Beijing 100190, China}
\address{School of Mathematical Sciences, University of Chinese Academy of Sciences, Beijing 100049, China}
\email{suyang@math.ac.cn}

\author{Jianqiang Yang}
\address{Department of Mathematics, Honghe University,  Yunan 661199, China}
\email{yangjq\_math@sina.com}

\date{}

\begin{document}

\begin{abstract}
We show some results of compact K\"ahler manifolds with elliptic homotopy type. In complex dimension $4$ we list the Hodge diamonds of compact K\"ahler manifolds with elliptic homotopy type. In general dimension we obtain a partial characterization of the Hodge diamonds.
\end{abstract}

\maketitle

\section{introduction}
From the rational homotopy theory point of view, nilpotent spaces are divided into two subclasses---rationally elliptic spaces and rationally hyperbolic spaces.  A closed simply connected manifold is \emph{rationally elliptic} if 
$$\dim \pi_*(M) \otimes \mathbb Q = \sum_{k\ge 2} \pi_k(X) \otimes \mathbb Q < \infty.$$
Rationally elliptic spaces form a special class of spaces -- there are much more rationally hyperbolic spaces -- however, they have nice algebraic-topological properties and are closely related to examples and conjectures in Riemannian geometry \cite{GH82}.  
Examples  of rationally elliptic manifolds include spheres, projective spaces, Grasssmannians, homogeneous spaces $G/H$, and fiber bundles with fiber and base simply-connected elliptic manifolds. A classification of rationally elliptic manifolds in low dimensions was given  in \cite{Herm17}. From such classifications one find more examples of rationally elliptic manifolds (see e.~g.~\cite[Proposition 4.6]{Herm17}).
 
The cohomology algebra of an elliptic space is subject to strong restrictions. Similar phenomena occur in complex geometry, where the cohomology algebras  of compact K\"ahler manifolds are endowed with rich structures. This comparison is one of the motivations for  the study of compact K\"ahler manifolds with rational elliptic homotopy type. Examples  of such manifolds include complex projective spaces and Grasssmannians, homogeneous spaces $G/P$, and holomorphic fiber bundles with fiber and base compact K\"ahler manifolds with elliptic homotopy type.  In \cite{AB} Amor\'os and Biswas considered compact K\"ahler manifolds with elliptic homotopy type in dimension $\le 3$. In dimension $2$ they classified compact K\"ahler surfaces with elliptic homotopy type (modulo the existence of fake quadrics); in dimension $3$ they gave a complete characterization of compact K\"ahler threefolds in terms of their Hodge diamonds. This characterization enables them to classify Fano threefolds with elliptic homotopy type, and found more compact K\"ahler threefolds with elliptic homotopy type which are neither homogenous nor fibrations (see \cite[Remark 4.9]{AB}) . 

The purpose of this paper is to extend our understanding of compact K\"ahler manifolds with elliptic homotopy type to dimension $4$ and higher. In dimension $4$, we obtain a list of possible Hodge diamonds and show that this is a partial characterization of compact K\"ahler fourfolds with elliptic homotopy type (Theorem \ref{thm:hodge}). As a corollary, we show that the Hodge conjecture is true for these manifolds, except for one possibility of Betti numbers. In general dimensions we  prove the vanishing of the geometric genus of compact K\"ahler manifolds with elliptic homotopy type (Theorem \ref{thm:dual}), and identify all complete intersections and smooth Fano weighted complete intersections with elliptic homotopy type (Theorem \ref{thm:ci}).

\begin{Theorem}\label{thm:hodge}
Let $X$ be a simply-connected compact K\"ahler manifold of complex dimension $4$. 
If $X$ is rationally elliptic, then 
\begin{enumerate}
\item the odd dimensional Betti numbers of $X$ are all zero;
\item the even dimensional  Hodge diamond of $X$ must be one of the following

\begin{eqnarray*}
\emph{(a) \ } 
\begin{array}{ccccc}
& & 1 & & \\
& 0 & 1 & 0 & \\
0 & 0 & 1 & 0 & 0\\
& 0 & 1 & 0 & \\
& & 1 & &
\end{array}
\ \ \ \emph{(b) \ } 
\begin{array}{ccccc}
& & 1 & & \\
& 0 & 1 & 0 & \\
0 & 0 & 2 & 0 & 0\\
& 0 & 1 & 0 & \\
& & 1 & &
\end{array} 
 \ \ \ \emph{(c) \ } 
\begin{array}{ccccc}
& & 1 & & \\
& 0 & 2 & 0 & \\
0 & 0 & 2 & 0 & 0\\
& 0 & 2 & 0 & \\
& & 1 & &
\end{array} 
\end{eqnarray*}

\begin{eqnarray*}
\emph{(d) \ }
\begin{array}{ccccc}
& & 1 & & \\
& 0 & 2 & 0 & \\
0 & 0 & 3 & 0 & 0\\
& 0 & 2 & 0 & \\
& & 1 & &
\end{array} 
\ \ \ \emph{(e) \ }
\begin{array}{ccccc}
& & 1 & & \\
& 0 & 3 & 0 & \\
0 & 0 & 4 & 0 & 0\\
& 0 & 3 & 0 & \\
& & 1 & &
\end{array}
\ \ \ \emph{(f) \ }
\begin{array}{ccccc}
& & 1 & & \\
& 0 & 4 & 0 & \\
0 & 0 & 6 & 0 & 0\\
& 0 & 4 & 0 & \\
& & 1 & &
\end{array} 
\end{eqnarray*}
\begin{eqnarray*}
\emph{(g) \ }
\begin{array}{ccccc}
& & 1 & & \\
& 1 & 2 & 1 & \\
0 & 2 & 2 & 2 & 0\\
& 1 & 2 & 1 & \\
& & 1 & &
\end{array}.
\end{eqnarray*}
\end{enumerate}

 All the Hodge diamonds except (g) can be realized by compact K\"ahler fourfolds with elliptic homotopy type. If $X$ has Hodge diamond (a) - (d), then $X$ is rationally  elliptic. If $X$ has Hodge diamond (a), then $X$ is real homotopy equivalent to $\mathbb C \mathrm P^4$; if $X$ has Hodge diamond (b), then $X$ is real homotopy equivalent to the quadric $V(2)$. 
\end{Theorem} 

\begin{remark}
It's shown in \cite[Theorem]{Wil86} that a complex projective fourfold with Hodge diamond (a) is either isomorphic to $\mathbb P^4$, or to another variety with pre-described invariants, whose existence is yet unknown.
\end{remark}

\begin{remark}
We don't have an example of a compact K\"ahler fourfold with elliptic homotopy type whose Hodge diamond is (g). In a subsequent paper it will be shown that this Hodge diamond cannot be realized by a projective manifold with elliptic homotopy type.
\end{remark}

\begin{remark}
It was asked by complex geometers whether the Chern number $\langle c_4(X), [X] \rangle $ is positive for a compact K\"ahler manifold $X$ of complex dimension $4$ (c.~f.~\cite[Problem 45]{YS}).  Theorem \ref{thm:hodge} (1) gives an affirmative answer when $X$ has elliptic homotopy type.

Rationally elliptic spaces with positive Euler characteristics are of particular interest. Some nice properties of these spaces are reflected in the Halperin conjecture (see \cite[\S 4.2, \S 9.7]{FOT}). It's known that the conjecture is true for compact K\"ahler manifolds.  In particular, by Theorem \ref{thm:hodge} and \cite{AB} the Halperin conjecture is true for compact K\"ahler manifolds of dimension $\le 4$. Therefore, by one geometric formulation of the conjecture, the space of self-homotopy equivalences of a compact K\"ahler manifold of dimension $\le 4$ with elliptic homotopy type is rationally homotopy equivalent to a product of odd dimensional spheres.
\end{remark}

%\begin{remark}
%The quadric $V(2)$ is biholomorphic to the Grassmann manifold $\mathrm{Gr}_2(\mathbb C^4)= U(4)/(U(2) \times U(2))$.
%\end{remark}

\begin{corollary}\label{cor:hc}
Let $X$ be a smooth complex projective variety of dimension $4$ with elliptic homotopy type, if $b_{2}(X) \ne 1$ or $b_{4}(X) \ne 2$, then the Hodge conjecture holds for $X$, i.~e., the Hodge classes of $X$ are all algebraic.
\end{corollary}

\begin{remark}
The remaining case with $b_{2}(X) =1$ and $b_{4}(X) = 2$  will be treated in a subsequence paper.
\end{remark}

For compact K\"ahler manifolds of elliptic homotopy type in general dimension, we also obtain some restrictions on their Hodge numbers. Recall that the geometric genus of a compact K\"ahler manifold $X$ of dimension $n$ is $p_g(X)=\dim H^n(X;\mathcal O_X)$. 

\begin{Theorem}\label{thm:dual}
Let $X$ be a compact K\"ahler manifold of complex dimension $n$ with elliptic homotopy type. Then
\begin{enumerate}
\item the geometric genus $p_g(X)=0$;
\item let 
$m$ be the largest integer such that $H^m(X,\mathcal O_X) \ne 0 $, then 
$$H^p(X, \mathcal O_X) \cong H^{m-p}(X, \mathcal O_X)$$ 
for all $p$.
\end{enumerate}
\end{Theorem}

\begin{remark}
Note that the geometric genus is a birational invariant. Therefore Theorem \ref{thm:dual} (1) also applies to compact K\"ahler manifolds which are birationally equivalent to a rationally elliptic one.
\end{remark}

\begin{corollary}\label{cor:chern}
Let $X$ be a compact K\"ahler manifold  with elliptic homotopy type. Then
\begin{enumerate}
\item $c_1(X) \ne 0$;
\item If $X$ is monotone and $H^p(X, \mathcal O_X) \ne 0$ for some $p>0$, then $X$ is of general type.
\end{enumerate}
\end{corollary}

Finally we identify rationally elliptic spaces among complete intersections and smooth Fano weighted complete intersections.

\begin{Theorem}\label{thm:ci}
The only complete intersections with elliptic homotopy type are the projective space $\mathbb P^n$ and the quadric $V(2)$. Besides $\mathbb P^n$ and $V(2)$, there are no more smooth well formed Fano weighted complete intersections with elliptic homotopy type. 
\end{Theorem}

\begin{remark}
The rational homotopy groups of complete intersections were characterized in \cite{Ne79}.
\end{remark}

The rational homotopy type of spaces has nice algebraic models (\cite{Quillen}, \cite{Sullivan}). The main idea of this paper is to analyze the algebraic models under the constraints provided by the K\"ahler package. We assume the reader is familiar with the fundamental facts about K\"ahler manifolds, and summarize the preliminaries of rational homotopy theory  in \S \ref{sec:app}.

\

\noindent \textbf{Acknowledgement.} We would like to thank Yifei Chen, Ming Fang, Baohua Fu and Ping Li for helpful discussions. 

\section{Preliminaries}\label{sec:app}
As the preliminaries for the proofs of the results in this paper, we summarize in this section the basic concepts and main results of rational homotopy theory, especially those related to elliptic spaces. For detailed information we refer to the books \cite{FOT} and \cite{FHT}.

\subsection{Minimal models and exponents}\label{sec:app1}
Let $V=\oplus_{i \ge 2} V^i$ be a graded vector space over $\mathbb Q$, $\wedge V$ be the tensor product of the polynomial algebra on $Q=V^{\mathrm{even}} = \oplus_{i} V^{2i}$ and the exterior algebra on $P=V^{\mathrm{odd}}=\oplus_{i} V^{2i+1}$, $d \colon \wedge V \to \wedge V$ be a linear map satisfying $d \circ d=0$ and $d(ab)=da \cdot  b + (-1)^{|a|} a \cdot db$. Then $(\wedge V,d)$ forms a  commutative graded differential algebra (cdga for short). A cdga $(\wedge V,d)$ is a \emph{minimal Sullivan algebra} if there is an ordering $<$ on a homogeneous basis $u_i$ such that $d(u_i) \in \wedge(u_j, j < i)$, and 
$d(V) \subset \wedge^{\ge 2} V$.  A minimal Sullivan algebra $(\wedge V,d)$ is \emph{elliptic} if both $V$ and $H^*(\wedge V,d)$ are finite dimensional. An elliptic minimal Sullivan algebra satisfies Poincar\'e duality, i.~e., $H^p(\wedge V,d) \cong H^{m-p}(\wedge V,d)$ where $m = \mathrm{max} \{ k \ | \ H^k(\wedge V,d) \ne 0 \} $  is the formal dimension of $(\wedge V,d)$. 

The rational homotopy theory of a simply-connected space $X$ is completely encoded in its \emph{rational minimal model} $(\wedge V, d)$. This is a minimal Sullivan algebra with $V_i \cong \mathrm{Hom}(\pi_i(X), \mathbb Q)$, and there is a quasi-isomorphism between $(\wedge V,d)$ and the cdga $(A_{\mathrm{PL}}(X), d)$, i.~e., a cdga map $\varphi \colon (\wedge V,d) \to A_{\mathrm{PL}}(X), d)$ such that 
$$\varphi_* \colon H^*(\wedge V,d) \to H^*(A_{\mathrm{PL}}(X), d)=H^*(X;\mathbb Q)$$ 
is an isomorphism.  The space $X$ is \emph{(rationally) elliptic} if its minimal model $(\wedge V,d)$ is elliptic.

Two spaces are rational homotopy equivalent if there is a continuous map between them inducing isomorphisms between the rational homotopy groups. This is equivalent to that there is an isomorphism between their rational minimal models. The real minimal model of $X$ is $(\wedge V,d) \otimes \mathbb R$. Two spaces are real homotopy equivalent if their real minimal models are isomorphic. 

Let $(\wedge V,d)$ be an elliptic minimal Sullivan, choose homogeneous basis elements $x_1, \cdots , x_q$ of $V^{\mathrm{even}}$, and $y_1, \cdots , y_r$ of $V^{\mathrm{odd}}$, denote the degree of $x_i$ by $2a_i$, and the degree of $y_j$ by $2b_j-1$. The sequence of integers $(a_1, \cdots, a_q; b_1, \cdots, b_r)$ is call the \emph{exponents} of $(\wedge V, d)$. If we arrange the exponents by descending order $a_1 \ge a_2 \ge \cdots \ge a_q$, $b_1 \ge b_2 \ge \cdots b_r$, then there are  inequalities 
\begin{equation}\label{eq:1}
 b_i \ge 2a_i,   \ \ \ \  1 \le i \le q.
 \end{equation}
 Furthermore, the exponents satisfy the following equalities and inequalities (\cite{FH79}). 

\begin{Theorem}\label{thm:exp}
The exponents of a simply connected, elliptic homotopy type minimal Sullivan algebra $(\wedge V,d)$ satisfy 
\begin{enumerate}
\item  $\dim V^{\mathrm{even}} = q \le r = \dim V^{\mathrm{odd}}$; 
\item $\sum_{i=1}^q 2a_i \le m$;
\item $\sum_{j=1}^r (2b_j-1) \le 2m-1$;
\item $m= 2(\sum_{j=1}^r b_j - \sum_{i=1}^a a_i)-(r-q)$;
\item $q=r$ if and only if all the odd Betti numbers vanish.
\end{enumerate}
when $m$ is the formal dimension of $(\wedge V,d)$.
\end{Theorem}

\subsection{Pure models}\label{sec:app2}
A minimal Sullivan algebra $(\wedge V,d)$ is \emph{pure} if 
$$d(V^{\mathrm{even}})=0 \ \ \textrm{and} \ \  d(V^{\mathrm{odd}}) \subset \wedge V^{\mathrm{even}}.$$ 
If one assumes further  $\dim V^{\mathrm{even}} = \dim V^{\mathrm{odd}}=k$, let $x_1, \cdots , x_k$ be a basis of $V^{\mathrm{even}}$, $y_1, \cdots, y_k$ be a basis of $V^{\mathrm{odd}}$, then $(\wedge V,d)$ is elliptic (i.~e.~$\dim H^*(\wedge V,d) < \infty$) if and only if $dy_1, \cdots ,dy_k$ is a regular sequence in $\mathbb Q[x_1, \cdots, x_k]$, i.~e., $dy_1$ in not a zero divisor in 
$\mathbb Q[x_1, \cdots, x_k]$ (which is automatically true), and $dy_i$ is not a zero divisor in 
$$\mathbb Q[x_1, \cdots, x_k]/(dy_1, \cdots, dy_{i-1})$$
for $i=2, \cdots, k$. In this case there is an isomorphism 
$$H^*(\wedge V,d) \cong \mathbb Q[x_1, \cdots, x_k]/(dy_1, \cdots, dy_k).$$ 
In general, for a minimal Sullivan algebra $(\wedge V,d)$,  there is an associated pure minimal Sullivan algebra $(\wedge V, d_{\sigma})$, where $d_{\sigma}$ is defined by 
$$d_{\sigma} (V^{\mathrm{even}})=0, \ \ d_{\sigma} \colon V^{\mathrm{odd}} \to \wedge V \to \wedge V^{\mathrm{even}},$$
where $\wedge V \to \wedge V^{\mathrm{even}}$ is the natural projection. There is a spectral sequence called the odd spectral sequence relating $H^*(\wedge V, d_{\sigma})$ to $H^*(\wedge V,d)$. As a consequence, $(\wedge V,d_{\sigma})$ and $(\wedge V,d)$ have the same cohomological dimension, therefore $(\wedge V,d)$ is elliptic if and only if $(\wedge V, d_{\sigma})$ is elliptic.  

\subsection{Formality}\label{sec:app3}
A minimal Sullivan algebra $(\wedge V, d)$ is \emph{formal} if it is quasi-isomorphic to its cohomology algebra with trivial differential $(H^*(\wedge V,d),0)$. In other words, there is a cdga map $(\wedge V,d) \to (H^*(\wedge V,d),0)$ inducing an isomorphism on cohomology. A space $X$ is \emph{formal} if its minimal model is formal.  Examples of formal spaces include spheres, Lie groups and compact K\"ahler manifolds. An algebraic criterion for formality is the following

\begin{Proposition}\label{prop:formal}
A minimal cdga $(\wedge V, d)$ is formal if and only if $V$ decomposes as a direct sum $V = C \oplus N$, with $dC=0$ and $d$ injective on $N$ such that every closed element in the ideal generated by $N$ is exact. 
\end{Proposition}

\subsection{Bigraded models}\label{sec:app4}
Let $(\wedge V,d)$ be the rational minimal model of a space $X$. The  complex minimal model of $X$ is $(\wedge V,d) \otimes \mathbb C$. For a compact K\"ahler manifold $M$, there is a bigradation on the complex minimal model which is compatible with the Hodge structure on $H^*(M;\mathbb C)$. 
\begin{Theorem}(\cite[Theorem 9.29]{FOT})\label{thm:big}
Let $M$ be a compact K\"ahler manifold. Then the complex minimal model $(\wedge V_{\mathbb C},d)$ of $M$ admits a bigradation such that the following properties hold.
\begin{enumerate}
\item Let $V_k$ be the subspace of elements of degree $k$ in $V_{\mathbb C}$. There is a decomposition of $V_k$ into $V_k = C_k \oplus N_k$, with $d|_{C_{k}}=0$, $d$ injective on $N_k$ and 
$$C_k = \oplus_{i+j=k}C_k^{i,j}$$

\item The bigradation is extended multiplicatively to $(\wedge V_{\mathbb C},d)$ and 
$$d(\wedge V_{\mathbb C})^{i,j} \subset (\wedge V_{\mathbb C})^{i,j}, \ \  (\wedge V_{\mathbb C})^{i,j} \cdot  (\wedge V_{\mathbb C})^{r,s} \subset  (\wedge V_{\mathbb C})^{i+r,j+s}.$$

\item The quasi-isomorphism $\rho \colon (\wedge V_{\mathbb C},d) \to (A^c(M),d)$ is compatible with the bigradation $A^c(M) = \oplus_{p,q}A^{p,q}$.

\item The bigradation induced on the cohomology of $(\wedge V_{\mathbb C},d)$ gives the pure Hodge structure of $H^*(M;\mathbb C)$. 
\end{enumerate}
\end{Theorem}

\section{Exponents}
In this section we study the exponents of minimal models of compact K\"ahler manifold with elliptic homotopy type. We use the same notations as in \S \ref{sec:app}.

\begin{Lemma}\label{lem:betti}
Let $X$ be a $2n$-dimensional elliptic space, with exponents $(a_1, \cdots, a_q;b_1, \cdots, b_r)$.
If $\sum_{i=1}^q a_i \ge n-2$ or $\sum_{j=1}^r b_j \ge 2n-1$, then all the odd dimensional Betti numbers of $X$ vanish.
\end{Lemma}
\begin{proof}
By Theorem \ref{thm:exp} (4),  $r-q$ is a non-negative even integer, if $r-q >0$, then $r-q \ge 2$. In this case the identity (Theorem \ref{thm:exp} (4)) $$2(\sum_{j=1}^r b_j -\sum_{i=1}^q a_i)-(r-q)=2n$$
together with the fact $b_j \ge 2$  imply 
\begin{eqnarray*}
\sum_{j=1}^r
b_j  & =  & \sum_{i=1}^q a_i + n + (r-q)/2 \\
      & \le & \frac{1}{2}\sum_{i=1}^q b_i +n + (r-q) -1 \\
      & \le & \frac{1}{2}\sum_{i=1}^q b_i + \frac{1}{2} \sum_{i=q+1}^r b_i +n-1 \\
      & = & \frac{1}{2}\sum_{j=1}^r b_j +n-1.
\end{eqnarray*}
From this one gets $\sum_{j=1}^r b_j \le 2n-2$. Therefore if $\sum_{j=1}^r b_j \ge 2n-1$ then $r-q=0$, and all the odd dimensional Betti numbers vanish by Theorem \ref{thm:exp} (5). 

Under the assumption $r-q>0$, the inequality $\sum_{i=1}^q a_i \ge n-2$ implies
$$\sum_{j=1}^r b_j \ge 2 \sum_{i=1}^q a_i + 2(r-q) \ge 2(n-2) + 4 = 2n.$$
By the result in the preceding paragraph one has $r-q=0$, a contradiction. Therefore if $\sum_{i=1}^q a_i \ge n-2$ then $r=q$, and all the odd dimensional Betti numbers of $X$ vanish. 
\end{proof}

\begin{Theorem}\label{thm:betti1}
Let $X$ be a compact K\"ahler manifold with elliptic homotopy type of complex dimension $n$, with exponents $(a_1, \cdots, a_q, b_1, \cdots b_r)$. 
\begin{enumerate}
\item If $\sum_{i=1}^q a_i \ge n-3$ or $\sum_{j=1}^r b_j \ge 2n-2$, then the odd dimensional Betti numbers of $X$ all vanish.

\item If $\sum_{i=1}^q a_i = n$, then $H^*(X;\mathbb Q)$ is generated by $H^2(X;\mathbb Q)$.
\end{enumerate}
\end{Theorem}

Theorem \ref{thm:hodge} (1) is a corollary of Theorem \ref{thm:betti1} (1), since $H^2(X;\mathbb Q) \ne 0$ for a compact K\"ahler manifold $X$, hence the $a$-exponents of its minimal model are nonempty and $\sum_{i=1}^q a_i \ge 1$.

\begin{proof}[Proof of Theorem \ref{thm:betti1}]
(1) By Lemma \ref{lem:betti} we only need to consider the case $\sum_{j=1}^r b_j = 2n-2$. We shall exclude this possibility. Note that the equality in Lemma \ref{lem:betti}  holds only when $r-q=2$, and in this case the exponents must be 
\begin{equation}\label{eq:ab}
b_1=2 a_1, \cdots , b_q=2 a_q=2, \  b_{q+1}=b_{q+2}=2.
\end{equation}
Let $(\wedge V, d)$ be the minimal model of $X$, denote 
$$Q^{2i}=V^{2i}, \ P^{2i+1}=V^{2i+1}, \ Q=\oplus Q^{2i} = V^{\mathrm{even}}, \ P=\oplus P^{2i+1} = V^{\mathrm{odd}}$$
 then $\dim P^{3} = \dim Q^{2}+2$. 

Compact K\"ahler manifolds are formal (see \S \ref{sec:app3}), by Proposition \ref{prop:formal} there is a decomposition $V=C \oplus N$ 
with $C=\ker d|_{V}$, and a surjection $\wedge C \to H^{*}(\wedge V, d)$.  Let $P_{0}=\ker d|_{P}$. If $P_{0}=0$, then $C=\ker d|_Q$, and $H^{*}(\wedge V, d)$ is generated by classes of even degrees. This implies that the odd dimensional Betti numbers vanish. But then Theorem \ref{thm:exp} (5) implies $r=q$, a contradiction. So one must have  $P_{0} \ne 0$.  

Now consider the associated pure minimal Sullivan algebra $(\wedge V, d_{\sigma})$, which is also elliptic (see \S \ref{sec:app2}). There is a decomposition
$$(\wedge V, d_{\sigma}) = (\wedge(Q \oplus P_{1}), d_{\sigma}) \otimes (\wedge P_{0}, 0),$$
in which $(\wedge(Q \oplus P_{1}), d_{\sigma})$ is an elliptic pure minimal Sullivan algebra with 
$\dim Q=\dim P_{1}=q$. Since the exponents of an elliptic minimal Sullivan algebra satisfy the inequalities (\ref{eq:1}) in \S \ref{sec:app}, and the exponenets of $(\wedge V,d)$ satisfy the equations (\ref{eq:ab}) above,  the vector space $P_0$ must be homogeneous of degree $3$. Let $x_{1}, \cdots, x_{q}$ and $y_{1}, \cdots, y_{q}$ be homogeneous basis of $Q$ and $P_{1}$, respectively,  then $d_{\sigma}y_{1}, \cdots , d_{\sigma}y_{q}$ is a regular sequence in $\mathbb Q[x_{1}, \cdots, x_{q}]$. This implies, if  $x_{1}, \cdots, x_{k}$ and $y_{1}, \cdots, y_{k}$ are basis of $Q^{2}$ and $P_{1}^{3}$, respectively, then $d_{\sigma}y_{1}, \cdots, d_{\sigma}y_{k}$ is a regular sequence in $\mathbb Q[x_{1}, \cdots, x_{k}]$. Therefore $(\wedge (Q^{2} \oplus P_{1}^{3}), d_{\sigma})$ is an elliptic minimal Sullivan algebra. Note that $d=d_{\sigma}$ on $P_1^3$, hence $(\wedge (Q^{2} \oplus P_{1}^{3}), d)$ is an elliptic minimal Sullivan algebra, whose formal dimension, by the k\"ahlerian condition, is $\ge 2n$. From this one deduces further by Theorem \ref{thm:exp} (4) that  $Q =Q^2$ and $P=P^{3}$, hence $(\wedge V ,d) \cong (\wedge (Q^2 \oplus P_1^3) ,d ) \otimes (\wedge P_0,0)$, a contradiction to the assumption that $X$ is K\"ahler.

Under the assumption $r-q>0$, $\sum_{i=1}^q a_i \ge n-3$ implies
$$\sum_{j=1}^r b_j \ge 2 \sum_{i=1}^q a_i + 2(r-q) \ge 2(n-3) + 4 = 2n-2$$
which  by the above discussion implies $r-q=0$, a contradiction. Therefore if $\sum_{i=1}^q a_i \ge n-3$ then $r=q$, hence all the odd dimensional Betti numbers of $X$ vanish.  

(2) By Theorem \ref{thm:exp} (4), and the facts  $b_j \ge 2a_j$, $b_j \ge 2$, if $\sum_{i=1}^q a_i =n$, then $r=q$ and $b_{i}=2a_{i}$ for all $i$. Then the discussion about the pure minimal Sullivan algebra $(\wedge V, d_{\sigma})$ above shows  $Q=Q^{2}$ and $P=P^{3}$. The surjection $\wedge Q^2 \to H^{*}(\wedge V,d)$ implies that  $H^{*}(\wedge V,d)$ is generated by $H^{2}(\wedge V,d)$.
\end{proof}

\begin{remark}
A closed $2n$-manifold $M$ is cohomologically symplectic (c-symplectic for short) if there is a cohomology class $\omega \in H^{2}(M;\mathbb R)$ such that $\omega^{n} \ne 0$. Theorem \ref{thm:betti1} (2) holds for c-symplectic manifolds with elliptic homotopy type. 
\end{remark}

Let $A$ be a graded commutative algebra, let $A^+= \oplus _{i>0} A^i$, define the \emph{indecomposables} of $A$ to be $Q(A)=A/(A^+ \cdot A^+)$  (c.~f.~\cite[p.60]{FOT}). Let $(\wedge V,d)$ be a formal minimal Sullivan algebra. By Proposition \ref{prop:formal}  the vector space $V$ has a decomposition $V=C \oplus N$ with $C=\ker d|_{V}$. Let $Q(H)$ be the indecomposables of  $H^*(\wedge V,d)$.

\begin{Lemma}\label{lem:ciso}
$Q(H)$ is isomorphic to $C$ as vector spaces.
\end{Lemma}

\begin{proof}
There is a homomorphism $C \to H^*(\wedge V,d)$, which is injective by the minimality of $(\wedge V,d)$, and  a  homomorphism $\wedge C \to H^*(\wedge V,d)$, which is surjective since any closed element in the ideal generated by $N$ is exact (Proposition \ref{prop:formal}). Let $Q(\wedge C)$ be the space of indecomposables of $\wedge C$, then clearly the map $C \to Q(C)$ is an isomorphism. The surjective homomorphism $\wedge C \to H^*(\wedge V,d)$ induces a surjective homomorphism   $Q(\wedge C) \to Q(H)$. The intersection of the image of $C$ in $H^*(\wedge V,d)$ with $H^+(\wedge V, d) \cdot H^+(\wedge V,d)$ is $0$, since for any $x \in C$, if $[x]=\sum [x_{i_1}] \cdots [x_{i_k}]$, then $x=\sum x_{i_1} \cdots x_{i_k} + dy$ for some $y$ and both $\sum x_{i_1} \cdots x_{i_k} $ and $dy$ are in $\wedge^{\ge 2} V$. This shows that the induced homomorphism $C \to Q(H)$ is an isomorphism.
\end{proof}

\begin{Lemma}\label{lem:ceven}
Let $(\wedge V,d)$ be the minimal model of a compact K\"ahler manifold, then for $k$ odd, $\dim  \ker(d|_{V^k})$ is even.
\end{Lemma}
\begin{proof}
It suffices to work with the complex minimal model. For a compact K\"ahler manifold $M$, the space of indecomposables $Q(H)$ inherits a bigradation from the Hodge decomposition on $H^*(M;\mathbb C)$, and a complex conjugation from the complex conjugation on $H^*(M, \mathbb C)$, compatible with the bigradation. Therefore for $k$ odd, by Lemma \ref{lem:ciso},
$$\dim \ker(d|_{V^k}) = \dim C^k = \dim Q(H)^k$$
is even. 
\end{proof}

\begin{Proposition}\label{prop:lambda}
Let $(\wedge V,d)$ be the minimal model of a compact K\"ahler manifold $X$ with elliptic homotopy type, of complex dimension $n$. Let $\lambda$ be the largest odd integer $k$ such that $\ker(d|_{V^{k}}) \ne 0$. Then $\lambda \le  n-2$. Equality holds only when $n=5$, $r=3$ and $q=1$.
\end{Proposition}
\begin{proof}
We work with the complex minimal model. Let $(a_1, \cdots, a_q; b_1, \cdots, b_r)$ be the exponents of $(\wedge V,d)$. If $r=q$, then all the odd dimensional Betti numbers of $X$ vanish, which implies $\ker (d|_{V^{\mathrm{odd}}})=0$. In the following we assume that $r>q$.

By the Hard Lefschetz Theorem, all cohomology classes of degree greater than $n$ are decomposable, whereas elements in $\ker (d|_V)$ give rise to indecomposable cohomology classes, therefore $\lambda \le n$.

When $n$ is even, suppose $\lambda =n-1$. Let $x \in \ker(d|_{V^{n-1}})$ be a nonzero element, $[x] \in H^{n-1}(X;\mathbb C)$ be the cohomology class. Then $[x]^2=0$ implies that there exists some $y \in (\wedge V)^{2n-3}$ such that $dy=x^2$. By degree reason,  the differential of any decomposable $y$ cannot produce the term $x^2$, therefore $y$ must be indecomposable, i.e.~$y \in V^{2n-3}$ and henceforth $\dim V^{2n-3} \ge 1$. Also by Lemma \ref{lem:ceven}, $\dim \ker(d|_{V^{n-1}}) \ge 2$. Therefore 
$$\sum_{j=1}^r b_j \ge 2 \cdot \frac{(n-1)+1}{2} + \frac{2n-3+1}{2}=2n-1.$$ 
It follows from Theorem \ref{thm:betti1} (1) that $r=q$, a contradiction.

When $n$ is odd, suppose $\lambda=n$, then by the same argument one gets $\sum_{j=1}^r b_j \ge 2n+1$, which implies by Theorem \ref{thm:betti1} (1) that $r=q$, a contradiction.  Now suppose $\lambda = n-2$. Then the same argument implies $\sum_{j=1}^r b_j  \ge 2n-3$. The assumption $r > q$ and Theorem \ref{thm:betti1} (1) imply $\sum_{j=1}^r b_j=2n-3$. Therefore one has 
$$b_{1}=b_{2}=(n-1)/2, \ b_{3}=n-2, \ r=3, \ q=1, \ a_{1}=1.$$ 
Solving the equation in Theorem \ref{thm:exp}(4) one gets $n=5$.
\end{proof}

Having proved the above lemmas and propositions, we are able to prove Theorem \ref{thm:ci} and Theorem \ref{thm:dual}. 
 
Recall that for a compact K\"ahler manifold $X$, the \emph{Hodge level} of $H^k(X;\mathbb C)$ is defined to be the largest number $| q-p|$ such that $p+q=k$ and $h^{p,q} \ne 0$, and to be $-\infty$ if $H^k(X;\mathbb C)=0$ (see \cite{Ra}).

\begin{proof}[Proof of Theorem \ref{thm:ci}]
From the fiber bundle $S^1 \to S^{2n+1} \to \mathbb P^n$ one easily sees that the projective space $\mathbb P^n$ is a rationally elliptic space. The quadric $Q^n$ is diffeomorphic to the homogeneous space $$SO(n+2)/(SO(n) \times SO(2)),$$
 hence is rationally elliptic. 

Let $X$ be a complete intersection of complex dimension $n$, different from the projective space $\mathbb P^{n}$, let $(\wedge V, d)$ be the minimal model of $X$, then the Lefschetz hyperplane theorem implies that 
$$\dim V^{2}=1, \ \ \ V^{3}= \cdots =V^{n-1}=0.$$ 

If $n$ is even, then the inequality $\sum a_{i} \le n$ (Theorem \ref{thm:exp} (2)) implies that $\dim V^{n} \le 1$. Therefore $\dim H^{n}(X;\mathbb C) \le 2$. The fact $h^{\frac{n}{2}, \frac{n}{2}} \ne 0$ and the symmetry $h^{p,n-p}=h^{n-p,p}$ imply that the Hodge level of $H^n(X;\mathbb C)$ is $0$.  Then by 
 \cite[Table 1]{Ra} $X$ is a hypersurface of degree $2$, i.e.,  the quadric $Q^{n}$.

If $n$ is odd, by Proposition \ref{prop:lambda} one has $\ker (d|_{V^{n}}) = 0$. Therefore $H^{n}(X;\mathbb C)=0$. By definition, the Hodge level of $H^n(X;\mathbb C)$ equals $-\infty$ and $X=Q^{n}$ by \cite[Table 1]{Ra}.

Now we consider weighted complete intersections. For the definitions and properties we refer to \cite{Dim}, \cite{PS20} and the references therein. Let $X$ be a smooth weighted complete intersection of dimension $n$ in a weighted projective space $\mathbb P(\mathbf{w})$, then Lefschetz theorem with $\mathbb Q$-coefficients still holds (\cite[Theorem B22]{Dim}), and the rational cohomolgy algebra of $\mathbb P(\mathbf{w})$ is isomorphic to the truncated polynomial algebra $\mathbb Q[\alpha]/(\alpha^N)$ with degree$(\alpha)=2$ (\cite[Proposition B13]{Dim}). Therefore by the same argument for ordinary complete intersections, the Hodge level of $H^n(X;\mathbb C)$ must be $0$ or $-\infty$ (This is called diagonal in \cite{PS20}). If we further assume that $X$ is smooth, well formed and Fano, then by \cite[Theorem 1.7]{PS20}, $X$ is $\mathbb P^n$, $V(2)$ or an intersection of two quadrics in $\mathbb P^{n+2}$. The last one is an ordinary complete intersection and by the result for complete intersections this is not rationally elliptic. 
\end{proof}

\begin{proof}[Proof of Theorem \ref{thm:dual}]
Let $(\wedge V,d)$ be the complex minimal model of $X$, where $V$ is endowed with the bigradation given  in Theorem \ref{thm:big}. Since the differential $d$ has bidegree $(0,0)$, the subalgebra $\wedge (\oplus_i V^{i,0})$ is closed under the differential $d$, and $(\wedge (\oplus_i V^{i,0}), d)$ is a minimal Sullivan algebra, whose cohomology $H^p((\wedge (\oplus_i V^{i,0}), d))$ is isomorphic to $H^{p,0}(X;\mathbb C)$. Therefore $(\wedge (\oplus_i V^{i,0}), d)$ is elliptic. If the formal dimension of $(\wedge (\oplus_i V^{i,0}), d)$ is $m$, then the cohomology groups satisfy the Poincar\' e duality $H^p(\wedge (\oplus_i V^{i,0}), d) \cong H^{m-p}(\wedge (\oplus_i V^{i,0}), d)$ (see \S \ref{sec:app1}), which implies $H^{p,0}(X) \cong H^{m-p,0}(X)$. In the following we show that $m < n$. 

Let $I$ be the ideal of $\wedge V$ generated by $\oplus_{i}(V^{i,0} + V^{0,i})$, then $I$ is closed under $d$, and the quotient algebra $\wedge V/I$ with the induced differential $\bar d$ is a minimal Sullivan algebra $(\wedge W, \bar d)$, with $W \cong \oplus_{i,j > 0} V^{i,j}$. The cohomology of $(\wedge W, \bar d)$ is finite dimensional, hence $(\wedge W, \bar d)$ is an elliptic minimal Sullivan algebra. (The finite dimensionality of $H^*(\wedge W,d)$ can be seen using the LS-categories: since there is a surjective homomomorphism $(\wedge V, d) \to (\wedge W, \bar d)$, the LS-categories satisfy $\mathrm{cat}(\wedge V, d) \ge \mathrm{cat}(\wedge W, \bar d)$ (\cite[Theorem 29.5]{FHT}), whereas the finiteness of the LS-category of a minimal Sullivan algebra is equivalent to the finite dimensionality of its cohomology (\cite[Proposition 32.4]{FHT}).) Now from the equation Theorem \ref{thm:exp} (4), we have $2n - 2m$ equals to the formal dimension of $(\wedge W, \bar d)$, which is positive by the K\"ahler condition. 
\end{proof}

\begin{proof}[Proof of Corollary \ref{cor:chern}]
A simply-connect compact K\"ahler manifold $X$ with $c_{1}(X)=0$ has trivial canonical line bundle $K_X$, therefore by  Serre duality  
$$H^{n}(X,\mathcal O_{X}) \cong H^{0}(X,K_{X}) = H^{0}(X, \mathcal O_{X}) \ne 0,$$ 
a contradiction to Theorem \ref{thm:dual}. Therefore $c_{1}(X) \ne 0$. 

Assume $X$ is monotone, i.e., $c_{1}(X) = \lambda \omega$ for some $\lambda \in \mathbb R$, where $\omega$ is the K\"ahler form. The possibility $\lambda >0$ is excluded since in this case $X$ is a Fano manifold, hence $H^{p}(X,\mathcal O_{X}) =0$ for all $p >0$, contradicting the assumption that $H^p(X, \mathcal O_X) \ne 0$ for some $p>0$. Therefore $\lambda$ must be negative and $X$ is of general type.
\end{proof}

\section{Hodge diamonds}
In this section we characterize the Hodge diamond of a compact K\"ahler manifold with elliptic homotopy type, especially in dimension $4$ and prove Theorem \ref{thm:hodge}. 

\begin{Theorem}\label{thm:hodgelevel}
Let $X$ be a compact K\"ahler manifold of complex dimension $n \ge 4$ with elliptic homotopy type and positive Euler characteristic $\chi(X)$. Then either the Hodge level of $H^{m}(X;\mathbb C)$ is $\le 0$ for all $m$, or there exists an $m \le (2n-1)/3$ such that the Hodge level of $H^m(X;\mathbb C)$ is positive.
\end{Theorem}
\begin{proof}
The assumption $\chi(X)>0$ implies that the odd dimensional Betti numbers of $X$ vanish, therefore if the Hodge level of $H^{m}(X;\mathbb C)$ is positive, then $m$ must be an even integer.  If the cohomology algebra $H^*(X;\mathbb C)$ is generated by $H^2(X;\mathbb C)$, then either the Hodge level of $H^{2}(X;\mathbb C)$ is positive, or the Hodge numbers $h^{p,q}=0$ for all $p \ne q$, therefore the conclusion holds. In the following we assume that $H^*(X;\mathbb C)$ is not generated by $H^2(X;\mathbb C)$, hence by Theorem \ref{thm:betti1} (2), the $a$-exponents of the minimal model satisfy $\sum_{i=1}^q a_i \le n-1$. 

Let $m$ be the smallest value such that there exist $p$, $q$ with $p+q=m$, $p \ne q$ and $h^{p,q} \ne 0$. Then the cohomology classes in $H^{p,q}(X)$ are indecomposable. Let $(\wedge V,d)$ be the complex minimal model of $X$,  $C=\ker(d|_{V})$. By Theorem \ref{thm:big} there is a bigradation on $C$,  denote by $C_m^{p,q}$ the subspace of bidegree $(p,q)$ with $p+q=m$. Then by Lemma \ref{lem:ciso} we have $\dim C_m^{p,q} \ge 1$, $\dim C^{q,p}_m \ge 1$. Let $x \in C_m^{p,q}$, $x' \in C_m^{q,p}$ be non-zero elements, with corresponding cohomology classes $[x] \in H^{p,q}(X)$, $[x'] \in H^{q,p}(X)$. If $m \le n/2$, then we are done with the proof. Assume $m > n/2$, by the Hard Lefschetz Theorem there are cohomology classes $\alpha$, $\beta$ and $\gamma$ such that 
$$[x][x'] = \alpha \omega^k, \ [x]^2=\beta \omega^k, \ [x]'^2 = \gamma \omega^k.$$
Therefore there are relations  
$$dy=x x' + \cdots, \ \ dz = x^2 + \cdots, \ \ dw = x'^2 + \cdots$$
By degree reason $y$, $z$ and $w$ must be indecomposable elements, and they are linearly independent since they have different  bigradation (note  that under the bigradation the differential $d$ has bidegree $(0,0)$):  
$$y \in V_{2m-1}^{m,m}, \ \ z \in V_{2m-1}^{2p,2q}. \ \ w \in V_{2m-1}^{2q,2p}.$$ 
So we have $\dim V_{2m-1} \ge 3$, then 
$$\sum_{j=1}^r b_j \ge 3 \cdot \frac{(2m-1)+1}{2} = 3m.$$
On the other hand, the inequality $\sum_{i=1}^q a_i \le n-1$ implies $\sum_{j=1}^r b_j \le 2n-1$.  Hence $m \le (2n-1)/3$.
\end{proof}

\begin{remark}
Theorem \ref{thm:hodgelevel} is also true when $n=2,3$. The $n=2$ case follows from an easy estimate of the exponents. The $n=3$ case was proved in \cite[Proposition 4.3]{AB} using deep results in complex geometry. 
\end{remark}

In the sequel we discuss the Hodge diamond of compact K\"ahler $4$-folds with elliptic homotopy type. By Theorem \ref{thm:hodge} (1) the odd dimensional Betti numbers all vanish, we only need to consider the even dimensional Hodge numbers. 

\begin{Proposition}\label{prop:notpure}
Let $X$ be a compact K\"ahler manifold of dimension $4$ with elliptic homotopy type. Then the only possible even dimensional Hodge diamond with positive Hodge level is 
$$\begin{array}{ccccc}
& & 1 & & \\
& 1 & 2 & 1 & \\
0 & 2 & 2 & 2 & 0\\
& 1 & 2 & 1 & \\
& & 1 & &
\end{array}$$
\end{Proposition}

\begin{proof}
We have shown in Theorem \ref{thm:hodge} (1) that the Euler characteristic $\chi(X) >0$. If some Hodge level of $X$ is positive, then by Theorem \ref{thm:hodgelevel} one must have $h^{2,0} \ge 1$, therefore $b_2(X) \ge 3$. The inequality $\sum_{i=1}^q a_i \le 4$ (Theorem \ref{thm:exp} (2)) implies $b_2(X)=3$ or $4$. In the following we analyze  these two possibilities.

If $b_2(X)=3$, then the exponents of $X$ are $a=(1,1,1)$, $b=(3,2,2)$ (c.f.~\cite[Lemma 5.1]{Herm17}) . Then from the construction of the minimal model $(\wedge V,d)$ one has 
\begin{eqnarray*}\dim H^4(X;\mathbb Q) &  = &\dim S^2H^2(X;\mathbb Q) - \dim (V^3/\ker d) +\dim(\ker d|_{V^{4}}) \\
 & = & 6-2+0 \\
 & = & 4
\end{eqnarray*}
Hence the even dimensional Hodge diamond of $X$ is 
$$\begin{array}{ccccc}
& & 1 & & \\
& 1 & 1 & 1 & \\
0 & 1 & 2 & 1 & 0\\
& 1 & 1 & 1 & \\
& & 1 & &
\end{array}$$
The signature of $X$ is $\mathrm{sign}(X)=\sum_{p,q}(-1)^{q}h^{p,q}=4$ (\cite{Hirz}), hence the intersection form of $X$ is positive definite. Let 
$$Q \colon H^4(X;\mathbb R) \times H^4(X;\mathbb R) \to \mathbb R$$
 be the intersection pairing, $\varphi \in H^{3,1}(X)$ be a non-zero element, $\alpha=\varphi + \bar \varphi$, $\beta=\varphi - \bar \varphi \in H^4(X;\mathbb R)$, then by the Hodge-Riemann bilinear relation $Q(\alpha, \alpha) < 0$, $Q(\alpha, \beta)=0$, and the paring of $\alpha$ with elements in $H^{2,2}(X)$ are all $0$.  Therefore the intersection form cannot be positive definite. This possibility is excluded.

If $b_{2}(X)=4$, then exponents are $a=(1,1,1,1)$ and $b=(2,2,2,2)$ (c.f.~\cite[Lemma 5.1]{Herm17}). Then 
$$\dim H^4(X;\mathbb Q) = \dim S^2H^2(X;\mathbb Q) - \dim V^3 = 6.$$
The Hard Lefschetz Theorem implies the even dimensional Hodge diamond has the following form
\begin{eqnarray*}
\text{(i) \ } 
\begin{array}{ccccc}
& & 1 & & \\
& 1 & 2 & 1 & \\
0 & 1 & 4 & 1 & 0\\
& 1 & 2 & 1 & \\
& & 1 & &
\end{array}
\ \ \ \text{(ii) \ } 
\begin{array}{ccccc}
& & 1 & & \\
& 1 & 2 & 1 & \\
0 & 2 & 2 & 2 & 0\\
& 1 & 2 & 1 & \\
& & 1 & &
\end{array} 
\end{eqnarray*}

For the first Hodge diamond, let $\omega, x \in H^{1,1}(X)$ be a basis, with $\omega$ being the K\"ahler form, $y \in H^{2,0}(X)$, $z \in H^{0,2}(X)$ be bases, then the Hard Lefschetz theorem implies 
$$H^4(X;\mathbb C) = S^2H^2(X;\mathbb C)/\langle y^2, z^2, xy-k\omega y, xz-l\omega z\rangle$$
for some $k, l \in \mathbb C$. 
Therefore the complex minimal model of $X$ must be $(\wedge V,d)=(\wedge (\omega, x ,y, z, a, b, c, e),d)$ with $\omega, x \in C_{2}^{1,1}$, $y \in C_{2}^{2,0}$, $z\in C_{2}^{0,2}$, $da=y^{2}$, $db=z^{2}$, $dc=xy - k \omega y$ and $de=xz - l \omega z$. Note that this is a pure minimal Sullivan algebra, its cohomology 
$$H^{*}(\wedge V,d) = \mathbb C[\omega, x, y, z]/(y^{2}, z^{2}, xy - k \omega y, xz - l \omega z)$$
is finite dimensional if and only if the sequence 
$$\{ y^{2}, z^{2}, xy - k \omega y, xz - l\omega z \} $$ 
is a regular sequence (see \S \ref{sec:app2}). But clearly $xy - k\omega y$ is a zero divisor in $\mathbb C[\omega, x, y, z]/(y^{2})$, a contradiction. Hence this possibility is ruled out.  

\end{proof}

\begin{Lemma}\label{lem:gen}
Let $X$ be a compact K\"ahler manifold of dimension $4$ with elliptic homotopy type. Then $H^{*}(X;\mathbb Q)$ is generated by $H^{2}(X; \mathbb Q)$ except for the case $b_{2}(X)=1$ and $b_{4}(X)=2$.
\end{Lemma}
\begin{proof}
Let $(a_1, \cdots, a_r;b_1, \cdots b_r)$ be the exponents of the minimal model of X. If $H^{*}(X;\mathbb Q)$ is not generated by $H^{2}(X;\mathbb Q)$, then by Theorem \ref{thm:betti1} (2), $\sum_{i=1}^{q}a_{i} \le 3$. On the other hand, if $H^{*}(X;\mathbb Q)$ is not generated by $H^{2}(X;\mathbb Q)$, then by Lemma \ref{lem:ciso}, there exists some $V^{2k}$ with $k \ge 2$ such that $\dim V^{2k} \ge 1$, i.e., there exists an $a_i \ge 2$. So the possible exponents are:  $a=(2,1)$, $b=(4,3)$ or $b=(5,2)$. The second possibility is ruled out by \cite[Lemma 5.1]{Herm17}. Therefore the only possible exponents are $a=(2,1)$, $b=(4,3)$, with $b_{2}(X)=1$ and $b_{4}(X)=2$.
\end{proof}

\begin{proof}[Proof of Theorem \ref{thm:hodge}]
The only possible Hodge diamond with positive Hodge level is given By Proposition \ref{prop:notpure}, which is the Hodge diamond (g) in Theorem \ref{thm:hodge}. In the following we list all possible Hodge diamonds with Hodge level $\le 0$.

Since $\chi(X)=0$, one has $r=q$.  Then  the possible exponents are given in \cite[Lemma 5.1]{Herm17}
$$\begin{array}{c|c|c|c|c|c|c|c|c}

a & (1) & (1,1) & (1,1) & (1,2) & (1,3 ) & (1,1,1) & (1,1,2) & (1,1,1,1) \\
\hline
b  & (5)  & (2,4) & (3,3) & (3,4) & (2,6)  & (2,2,3) & (2,2,4) & (2,2,2,2)

\end{array}$$

For the exponents $a=(1,3)$, $b=(2,6)$,  the vanishing of $b_{3}(X)$ implies that the differential of a generator of $V^{3}$ is a non-zero multiple of $\omega^{2}$, where $\omega \in V^2$ represents the K\"ahler form. This is not allowed. 

For the exponents $a=(1,1,2)$, $b=(2,2,4)$,   the vanishing of $b_{3}(X)$ implies $d \colon V^3 \to S^2V^2$ is injective. From the construction of the minimal model one has 
$$H^4(X;\mathbb Q) \cong S^2H^2(X;\mathbb Q)/d(V^3) \oplus \ker(d|_{V^4}).$$
On the other hand from the formula in \cite{FH79}
$$\sum_{i=1}^{8} \dim H^{i}(X; \mathbb Q) = \prod_{i=1}^{q}b_{i}/\prod_{i=1}^{q}a_{i}$$
one computes the Betti number $b_4(X)=2$. This shows $\ker(d|_{V^4})\ne 0$ and therefore $H^4(X;\mathbb Q)$ is not generated by $H^2(X;\mathbb Q)$. Hence this possibility is excluded by Lemma \ref{lem:gen}. 

For the remaining possibilities, one may use the formula
$$\sum_{i=1}^{8} \dim H^{i}(X; \mathbb Q) = \prod_{i=1}^{q}b_{i}/\prod_{i=1}^{q}a_{i}$$
to compute the Betti number $b_{4}(X)=h^{2,2}$. The results are presented in the following table

$$\begin{array}{c|c|c|c|c|c|c}

b_2(X) & 1  & 1   & 2  & 2 &  3 & 4  \\
\hline
b_4(X)  & 1  & 2  & 2 & 3 & 4  & 6 

\end{array}$$
From this one gets the Hodge diamonds (a) - (f).

For the Hodge diamond (a), the k\"ahlerian condition implies that $H^*(X;\mathbb R)$ is isomorphic to $\mathbb R[\omega]/(\omega^5)$. Compact K\"ahler manifolds are formal, which means that their cohomology algebra $H^*(X;\mathbb Q)$ determines the rational homotopy type. Hence $X$ is real homotopy equivalent to $\mathbb C \mathrm P^4$, hence is rationally elliptic.

For the Hodge diamond (b), let $\omega \in H^2(X;\mathbb R)$ be the class of the K\"ahler form, $H^4(X;\mathbb R)$ has basis $\omega^2$, $x$. The Hard Lefschetz theorem gives a relation $\omega \cdot x = a \omega^3$ for some $a \in \mathbb R$. After a basis change $x \to x + a \omega^2$ one has the relation $\omega \cdot x =0$. Now we have $x^2 = \lambda \omega^4$ for some $\lambda \in \mathbb R$. The signature of $X$ equals $\sum(-1)^qh^{p,q}=2$, therefore $\lambda >0$. Rescaling $\omega \to \lambda^{-1/4} \omega$, one has $x^2 =\omega^4$. It is easy to verify that  the real cohomology algebra $H^*(X;\mathbb R)$ is isomorphic to $\mathbb R[\omega, x]/(\omega \cdot x, x^2-\omega^4)$, which is isomorphic to the real cohomology algebra of the quadric hypersurface $V(2)$.
Hence $X$ is real homotopy equivalent to $V(2)$. (This example is missing in \cite[Proposition 5.1]{Herm17}.)

Next we use the same idea as in \cite{AB} to analyze the Hodge diamonds (c) and (d). For the Hodge diamond (c), let $\omega \in H^2(X;\mathbb R)$ be the class of the K\"ahler form, $x \in H^2(X;\mathbb R)$ be a primitive class, i.~e., $\omega^3 \cdot x=0$. By the Hard Lefschetz theorem $\{\omega^2, \omega x\}$ is a basis of $H^4(X;\mathbb R)$. Therefore there is a relation 
$$x^2 = \alpha \omega^2 + \beta \omega x$$
for some $\alpha, \beta \in \mathbb R$. By the Hodge-Riemann bilinear relation 
$$0 < -Q(x, \bar x) = -\langle \omega^2 x^2, [X] \rangle = -\alpha \langle \omega^4, [X] \rangle.$$
Therefore by a rescaling of $\omega$ we may assume $\alpha=-1$. Let 
$$p_1(\omega,x)=x^2+\omega^2 -\beta \omega x, \ \ \ p_2(\omega,x)=\omega^3 x,$$ 
then it's easy to check by comparing the dimensions in each degree that $H^*(X;\mathbb R) \cong \mathbb R[\omega,x]/(p_1, p_2)$. We construct a minimal Sullivan algebra $(\wedge V,d)$ as follows: $V^2$ is generated by $\omega$ and $x$, with $d \omega=dx=0$, $V^3$ is generated by $y$ with $dy=p_1(\omega,x)$, $V^7$ is generated by $z$ with $dz=p_2(\omega,x)$. Then $(\wedge V,d)$ is a pure Sullivan algebra, with $\dim V^{\mathrm{even}}=\dim V^{\mathrm{odd}}$. If $(p_1,p_2)$ is a regular sequence in $\mathbb R[\omega, x]$, then $(\wedge V,d)$ is rationally elliptic and $H^*(\wedge V,d) \cong \mathbb R[\omega,x]/(p_1, p_2)$ (see \S \ref{sec:app2}). Notice that the ideal $(\omega, x)^5$ is contained in the ideal $( p_1, p_2)$, and $(\omega, x)$ is a maximal ideal, therefore $(\omega, x)$ is the radical of $( p_1, p_2)$, and this implies that the sequence $\{p_1, p_2\}$ is a regular sequence.  This shows that there is a quasi-isomorphism $(\wedge V,d) \to H^*(X;\mathbb R)$, mapping $\omega$ and $x$ to cohomology classes with the same name, and $y$, $z$ to $0$. Hence $(\wedge V,d)$ is the real minimal model of $X$, and $X$ is rationally elliptic, with $\pi_2(X) \otimes \mathbb Q \cong \mathbb Q^2$, $\pi_3(X) \otimes \mathbb Q \cong \mathbb Q$, $\pi_7(X) \otimes \mathbb Q \cong \mathbb Q$, and $\pi_i(X) \otimes \mathbb Q=0$ for other $i$. 

For the Hodge diamond (d), let $\omega \in H^2(X;\mathbb R)$ be the class of the K\"ahler form, $x \in H^2(X;\mathbb R)$ be a primitive class, i.~e., $\omega^3 \cdot x=0$. Since $H^4(X;\mathbb R)$ is generated by $H^2(X;\mathbb R)$ (Lemma \ref{lem:gen}),  $\{\omega^2, \omega x, x^2\}$ is a basis of $H^4(X;\mathbb R)$. By the Hard Lefschetz theorem, $\{\omega^3, \omega^2 x\}$ is a basis of $H^6(X;\mathbb R)$. We have relations
\begin{eqnarray}\label{eqn:2}
x^3& = & \alpha \omega^3 + \beta \omega^2 x 
\end{eqnarray}
\begin{eqnarray}\label{eqn:3}
\omega x^2 & = & \gamma \omega^3 + \delta \omega^2 x
\end{eqnarray}
for some $\alpha, \beta, \gamma, \delta \in \mathbb R$. The Hodge-Riemann bilinear relation implies 
$$Q(x, \bar x)=\langle \omega^2 x^2, [X] \rangle =\gamma \langle \omega^4, [X] \rangle <0$$
After a rescaling of $\omega$ we may assume that $\gamma=-1$. Then 
$$\omega^2 x^2=-\omega^4 + \delta \omega^3 x=-\omega^4,$$ 
and from (\ref{eqn:2}) we have
$$\omega x^3=\alpha \omega^4 + \beta \omega^3 x =\alpha \omega^4,$$
from (\ref{eqn:3}) we have $$\omega x^3 = -\omega^3 x + \delta \omega^2 x^2 =-\delta \omega^4.$$
Therefore $\delta=-\alpha$. This implies that $x^2+\omega^2 + \alpha \omega x$ is a primitive class in $H^4(X;\mathbb R)$, hence the Hodge-Riemann bilinear relation implies
$$0 < \langle (x^2 + \omega^2 +\alpha \omega x)^2, [X] \rangle = \alpha^2 - \beta -1.$$
Now we have seen that the cohomology algebra $H^*(X;\mathbb R)$ is generated by $\omega$ and $x$, and there are relations
$$x^3=\alpha \omega^3 -\beta \omega^2 x, \ \ \ \omega x^2 = -\omega^3 -\alpha \omega^2 x$$
for some $\alpha, \beta \in \mathbb R$ such that $\beta < \alpha^2-1$. Let 
$$p_1(\omega, x)= x^3-\alpha \omega^3 -\beta \omega^2 x, \ \ p_2(\omega,x)=\omega x^2  +\omega^3 +\alpha \omega^2 x,$$ 
then by comparing the dimensions in each degree one shows that $H^*(X;\mathbb R)$ is isomorphic to $\mathbb R[\omega, x]/( p_1, p_2)$. (This is seen as follows: from the equations
$$\omega x^3=\alpha \omega^4 +\beta \omega^3 x$$
and 
$$\omega x^3 = -\omega^3 x -\alpha \omega^2 x^2 = -\omega^3 x -\alpha \omega (-\omega^3-\alpha \omega^2 x)=\alpha \omega^4 +(\alpha ^2-1) \omega^3 x$$
we see $\omega^3 x=0$ since $\alpha^2-1 \ne \beta$. Then one readily gets
$$\omega^2 x^2 = -\omega^4, \ \ \omega x^3 = \alpha \omega^4, \ \ x^4=-\beta \omega^4, \ \ \omega^5=0.)$$
We construct a minimal Sullivan algebra $(\wedge V,d)$, with $V^2$ being generated by $\omega$ and $x$, $d\omega=dx=0$, $V^5$ being generated by $y$ and $z$, $dy=p_1(\omega,x)$, $dz=p_2(\omega,x)$. This is a pure Sullivan algebra with $\dim V^{\mathrm{even}}=\dim V^{\mathrm{odd}}$. The radical of the ideal $( p_1, p_2)$ is $(\omega, x)$, since $(\omega,x)^5 \subset (p_1, p_2)$. Therefore $\{p_1, p_2\}$ is a regular sequence and $H^*(\wedge V,d) \cong \mathbb R[\omega,x]/(p_1, p_2)$.  Therefore $(\wedge V,d)$ is the real minimal model of $X$, and $X$ is rationally elliptic, $\pi_2(X) \otimes \mathbb Q \cong \mathbb Q^2$,  $\pi_5(X) \otimes \mathbb Q \cong \mathbb Q^2$, $\pi_i(X) \otimes \mathbb Q=0$ for $i \ne 2,5$.

Finally it is easy to verify that the Hodge diamonds (a)-(f) can be realized by $\mathbb C \mathrm P^4$, the quadric hypersurface $V(2)$, $\mathbb C \mathrm P^1 \times  \mathbb C \mathrm P^3$, $\mathbb C \mathrm P^2 \times \mathbb C \mathrm P^2$, $\mathbb C \mathrm P^1 \times \mathbb C \mathrm P^1 \times \mathbb C \mathrm P^2$ and $(\mathbb C \mathrm P^1)^4$, respectively. 
\end{proof}

\begin{proof}[Proof of Corollary \ref{cor:hc}]
Let $X$ be a projective manifold with Hodge diamond (a), (c)-(f), since $H^*(X;\mathbb Q)$ is generated by $H^2(X;\mathbb Q)$ and the Hodge diamond has Hodge level $\le 0$, the Lefschetz $(1,1)$-theorem implies that every Hodge class is a rational linear combination of cohomology classes of subvarieties. For the Hodge diamond (g), by rescaling the K\"ahler form $\omega$ we may assume that the cohomology class $[\omega]$ is rational,  $[\omega] \in H^2(X;\mathbb Q)$. Since $h^{1,1}=h^{2,2}=2$,  the Hard Lefschetz theorem implies that 
$$H^2(X;\mathbb Q) \cap H^{1,1}(X) \to H^4(X;\mathbb Q) \cap H^{2,2}(X), \ \ \ \alpha \mapsto \alpha \wedge \omega$$
is an isomorphism, i.~e., $\mathrm{Hdg}^2(X)$ is generated by $\mathrm{Hdg}^1(X)$. Hence the Hodge conjecture holds.
\end{proof}

\end{document}